\documentclass[12pt]{amsart}
\usepackage{amscd}
\usepackage{amsmath}
\usepackage{amssymb}
\usepackage{amsthm}
\usepackage{graphicx}
\usepackage{enumerate}

\newtheorem{theorem}{Theorem}[section]
\newtheorem{proposition}[theorem]{Proposition}
\newtheorem{lemma}[theorem]{Lemma}
\newtheorem{corollary}[theorem]{Corollary}

\theoremstyle{definition}

\newtheorem{remark}[theorem]{Remark}

\def\nR{{\mathbb{R}}}

\DeclareMathOperator{\pd}{pd}

\DeclareMathOperator{\height}{height}

\DeclareMathOperator{\reg}{reg}

\DeclareMathOperator{\Soc}{Soc}

\DeclareMathOperator{\link}{link}

\def\eb{{\bold e}}

\numberwithin{equation}{section}

\begin{document}

\title[Cameron--Walker graphs]
{Algebraic study on Cameron--Walker graphs}

\author[]{Takayuki Hibi}
\address[]{Department of Pure and Applied Mathematics,
Graduate School of Information Science and Technology,
Osaka University,
Toyonaka, Osaka 560-0043, Japan}
\email{hibi@math.sci.osaka-u.ac.jp}
\author[]{Akihiro Higashitani}
\address[]{
Department of Mathematics, Graduate School of Science, 
Kyoto University, Kitashirakawa-Oiwake cho, Sakyo-ku, Kyoto 606-8502, Japan}
\email{ahigashi@math.kyoto-u.ac.jp}
\author[]{Kyouko Kimura}
\address[]{Department of Mathematics, Graduate School of Science, 
Shizuoka University,
836 Ohya, Suruga-ku, Shizuoka 422-8529, Japan}
\email{skkimur@ipc.shizuoka.ac.jp}
\author[]{Augustine B. O'Keefe}
\address[]{Department of Mathematics,
University of Kentucky, 
Lexington, KY 40506, U.S.A}
\email{augustine.okeefe@uky.edu}


\begin{abstract}
Let $G$ be a finite simple graph on 
$[n]$ and $I(G) \subset S$ the edge ideal of $G$,
where
$S = K[x_{1}, \ldots, x_{n}]$ is the polynomial ring 
over a field $K$.
Let $m(G)$ denote the maximum size of matchings of $G$
and $im(G)$ that of induced matchings of $G$. 
It is known that 
$im(G) \leq \reg(S/I(G)) \leq m(G)$, where
$\reg(S/I(G))$ is the Castelnuovo--Mumford regularity of $S/I(G)$.
Cameron and Walker succeeded in classifying the finite connected simple graphs
$G$ with $im(G) = m(G)$.
We say that a finite connected simple graph $G$ is 
a Cameron--Walker graph if $im(G) = m(G)$ and if
$G$ is neither a star nor a star triangle. 
In the present paper, we study Cameron--Walker graphs
from a viewpoint of commutative algebra.
First, we prove that a Cameron--Walker graph $G$ is unmixed 
if and only if $G$ is Cohen--Macaulay and classify all
Cohen--Macaulay Cameron--Walker graphs.
Second, we prove that there is no Gorenstein Cameron--Walker graph. 
Finally, we prove that every Cameron--Walker graph
is sequentially Cohen--Macaulay. 
\end{abstract}
\maketitle
\section*{Introduction}
Recently, edge ideals of finite simple graphs have been studied by many authors
from viewpoints of computational commutative algebra and combinatorics; 
see \cite{HaTuylsurvey, MV, Villarreal}, and their references. 

Let $[n] = \{ 1, \ldots, n\}$ be a vertex set and
$G$ a finite simple graph on $[n]$
with $E(G)$ its edge set.  (A simple graph
is a graph with no loop and no multiple edge.)
Let $S = K[x_{1}, \ldots, x_{n}]$ denote 
the polynomial ring in $n$ variables over a field $K$. 
The {\em edge ideal} (\cite[p.\  156]{HHgtm260}) of $G$ is the monomial ideal
$I(G)$ of $S$ generated by those monomials $x_{i}x_{j}$ with
$\{ i, j \} \in E(G)$, viz.,
\[
I(G) = ( \, x_{i}x_{j} \, : \, \{ i, j \} \in E(G) \, ) \, \subset \, S.
\]
One of the research topics on $I(G)$ is the computation of
the Castelnuovo--Mumford regularity $\reg(S/I(G))$ 
(\cite[p.\  48]{HHgtm260}) of $S/I(G)$ 
in terms of the invariants of $G$. 

Recall that a subset $M$ of $E(G)$ is a {\em matching} of $G$
if, for $e$ and $e'$ belonging to $M$ with $e \neq e'$, one has $e \cap e' = \emptyset$.
The {\em matching number} $m(G)$ of $G$ is the maximum size of matchings of $G$.  
A matching $M$ of $G$ is called an {\em induced matching} of $G$ if, 
for $e$ and $e'$ belonging to $M$ with $e \neq e'$, 
there is no $f \in E(G)$
with $f \cap e \neq \emptyset$ and $e' \cap f \neq \emptyset$.
Let $im(G)$ denote the maximum size of induced matchings of $G$. 

For example, if $G = K_{n}$, a complete graph on $[n]$, then
$im(G) = 1$ and $m(G) = \lfloor n/2 \rfloor$.  
If $G = K_{m,n}$ ($m \leq n$), a complete bipartite graph
with vertex partition $[n] \sqcup [m]$, 
then $im(G) = 1$ and $m(G) = m$.
If $G$ is the Petersen graph, then 
$im(G) = 3$ and $m(G) = 5$.

It is known (\cite[Lemma 2.2]{Katzman} and \cite[Theorem 6.7]{HaTuyl08})
that
\[
im(G) \leq \reg(S/I(G)) \leq m(G).
\]
One has $\reg(S/I(G)) = im(G)$
for, e.g., chordal graphs,
unmixed bipartite graphs 
and sequentially Cohen--Macaulay bipartite graphs
(\cite{HaTuyl08, Kummini, VanTuyl}; 
see also \cite{FHVT, MMCRTY, Woodroofe10, Zheng}). 

Cameron and Walker (\cite[Theorem 1]{CW}) gave a classification 
of the finite connected simple graphs
$G$ with $im(G) = m(G)$, although there is a mistake; 
see Remark \ref{chuu} below. 
By modifying their result slightly, we see that 
a finite connected simple graph $G$ satisfies 
$im(G) = m(G)$ if and only if $G$ is 
one of the following graphs: 
  \begin{itemize}
  \item a star; 
  \item a star triangle; 
  \item a finite graph consisting of a connected bipartite graph 
    with vertex partition 
    $[n] \sqcup [m]$ such that there is at least one leaf edge 
    attached to each vertex
    $i \in [n]$ and that there may be possibly some pendant triangles 
    attached to each vertex $j \in [m]$. 
  \end{itemize}
Here a star triangle is a graph 
joining some triangles at one common vertex, 
e.g., the graph on $\{1,\ldots,7\}$ with the edges 
$\{1,2\}$, $\{1,3\}$, $\{2,3\}$, $\{1,4\}$, $\{1,5\}$, $\{4,5\}$, $\{1,6\}$, $\{1,7\}$, $\{6,7\}$ is a star triangle. 
A {\em leaf} is a vertex of degree $1$ and a {\em leaf edge} is an edge 
meeting a leaf. Also a {\em pendant triangle} is a triangle 
whose two vertices have degree $2$ 
and the rest vertex has degree more than $2$. 
We say that a finite connected simple graph $G$ is 
a {\em Cameron--Walker graph} if $im(G) = m(G)$ and if 
$G$ is neither a star nor a star triangle. For example, 

\begin{center}
  \begin{picture}(300,150)(10,-10)
  \put(100,40){\circle*{5}}
  \put(150,40){\circle*{5}}
  \put(200,40){\circle*{5}}
  \put(250,40){\circle*{5}}
  \put(125,90){\circle*{5}}
  \put(175,90){\circle*{5}}
  \put(225,90){\circle*{5}}
  \put(85,120){\circle*{5}}
  \put(115,120){\circle*{5}}
  \put(145,120){\circle*{5}}
  \put(175,120){\circle*{5}}
  \put(225,120){\circle*{5}}
  \put(60,20){\circle*{5}}
  \put(90,0){\circle*{5}}
  \put(110,0){\circle*{5}}
  \put(140,20){\circle*{5}}
  \put(185,10){\circle*{5}}
  \put(215,10){\circle*{5}}
  \put(100,40){\line(1,2){25}}
  \put(150,40){\line(1,2){25}}
  \put(150,40){\line(3,2){75}}
  \put(200,40){\line(-3,2){75}}
  \put(200,40){\line(-1,2){25}}
  \put(200,40){\line(1,2){25}}
  \put(250,40){\line(-1,2){25}}
  \put(125,90){\line(-4,3){40}}
  \put(125,90){\line(-1,3){10}}
  \put(125,90){\line(2,3){20}}
  \put(175,90){\line(0,1){30}}
  \put(225,90){\line(0,1){30}}
  \put(100,40){\line(-2,-1){40}}
  \put(100,40){\line(-1,-4){10}}
  \put(60,20){\line(3,-2){30}}
  \put(100,40){\line(1,-4){10}}
  \put(100,40){\line(2,-1){40}}
  \put(110,0){\line(3,2){30}}
  \put(200,40){\line(-1,-2){15}}
  \put(200,40){\line(1,-2){15}}
  \put(185,10){\line(1,0){30}}
\end{picture}

{\bf Figure 1} (Cameron--Walker graph)

\bigskip

\end{center}
is a Cameron--Walker graph.

\begin{remark}\label{chuu}
  The original result of Cameron--Walker \cite[Theorem 1]{CW} 
  claimed ``a triangle'' instead of ``a star triangle'' 
  in the above classification. 

  \par
  The reason why we claimed differently is that the ``only if'' part of \cite[Theorem 1]{CW} is a little wrong; 
  concretely, in the second paragraph in the proof of Theorem 1 (Only if) 
  \cite[p.\  54]{CW}. 
  Their argument asserted that 
  when we delete all pendant triangles of $G$, 
  we get a connected bipartite graph $H$. 
  However $H$ is possibly an isolated vertex; this case was forgotten, 
  and in such case, $G$ should be a star triangle. 
  Indeed a star triangle $G$ also satisfies $im(G)=m(G)$. 
\end{remark}

In the present paper, we study Cameron--Walker graphs
from a viewpoint of commutative algebra.
One of the main problems is which Cameron--Walker graphs
are Cohen--Macaulay.

Let $G$ be a finite simple graph on $[n]$.  
A {\em vertex cover} of $G$ is a subset $C$ of $[n]$ such that 
$C \cap e \neq \emptyset$ for all $e \in E(G)$.
A vertex cover $C$ is called {\em minimal} 
if no proper subset of $C$ is a vertex cover of $G$. 
A finite simple graph $G$ is called {\em unmixed} if all
minimal vertex covers of $G$ have the same cardinality. 
A finite simple graph $G$ is {\em Cohen--Macaulay} if
$S/I(G)$ is Cohen--Macaulay. 
Every Cohen--Macaulay graph is unmixed (\cite[Lemma 9.1.10]{HHgtm260}).
A finite simple graph $G$ is called 
{\em vertex decomposable}, {\em shellable} or {\em sequentially Cohen--Macaulay} 
if the simplicial complex $\Delta({\overline G})$ 
is vertex decomposable, shellable or sequentially Cohen--Macaulay, respectively (\cite[p.\ 144]{HHgtm260}),
where ${\overline G}$ is the complementary graph of $G$  (\cite[p.\ 153]{HHgtm260}) 
and $\Delta({\overline G})$ is the clique complex (\cite[p.\ 155]{HHgtm260}) of ${\overline G}$.  
Every vertex decomposable graph is shellable (\cite[Theorem 11.3]{BW}) and 
every shellable graph is sequentially Cohen--Macaulay (\cite[Corollary 8.2.19]{HHgtm260}). 
Note that $G$ is unmixed if and only if $\Delta(\overline{G})$ is pure. 

In Theorem \ref{Kimura}, 
we prove that, for a Cameron--Walker graph $G$, 
the following five conditions are equivalent: 
  \begin{itemize}
  \item $G$ is unmixed; 
  \item $G$ is Cohen--Macaulay; 
  \item $G$ is unmixed and shellable; 
  \item $G$ is unmixed and vertex decomposable; 
  \item $G$ consists of a connected bipartite graph with vertex partition 
    $[n] \sqcup [m]$ such that 
    there is exactly one leaf edge attached to each vertex
    $i \in [n]$ and that there is exactly one pendant triangle
    attached to each vertex $j \in [m]$.
  \end{itemize}

When $G$ is Cohen--Macaulay, we call the Cohen--Macaulay type of $S/I(G)$ 
the Cohen--Macaulay type of $G$. 
A finite simple graph $G$ is called {\em Gorenstein} if $S/I(G)$ is Gorenstein. 
Note that $G$ is Gorenstein if and only if $G$ is Cohen--Macaulay 
with Cohen--Macaulay type 1. 
We also consider the problem which Cameron--Walker graphs are Gorenstein. 
The answer is that there is no Gorenstein Cameron--Walker graph. 
We draw this conclusion by computing 
the Cohen--Macaulay types of 
all Cohen--Macaulay Cameron--Walker graphs (Theorem \ref{GorCWgraph}). 

\par
In addition, we also consider the problem which Cameron--Walker graphs are 
sequentially Cohen--Macaulay. 
In Theorem \ref{seqCM}, we prove that every Cameron--Walker graph is 
sequentially Cohen--Macaulay. Actually, it is vertex decomposable 
and thus, shellable. 
We also give a shelling for a Cameron--Walker graph 
whose supporting connected bipartite graph is a complete bipartite graph 
(Proposition \ref{shellable}).

\bigskip

\noindent
\textbf{Acknowledgments}:
The authors are grateful to Seyed Amin Seyed Fakhari, Russ Woodroofe 
and Siamak Yassemi for giving them helpful comments on the first version of this paper. 
The authors would also like to appreciate the anonymous referee 
for suggesting the crucial ideas of the proof of Theorem \ref{GorCWgraph}. 

\par
The second author is partially supported by JSPS Research Fellowship for Young Scientists. 
The third author is partially supported by JSPS Grant-in-Aid 
for Young Scientists (B) 24740008.

\section{Cohen--Macaulay Cameron--Walker graphs}
In this section, we classify all Cohen--Macaulay Cameron--Walker graphs. 
It will turn out that such graphs are of the form 
that a connected bipartite graph each of whose vertex has 
exactly one leaf edge ($K_2$) or exactly one pendant triangle ($K_3$). 
We will first prove that more generally, for a finite simple graph $G$, 
the graph obtained by attaching complete graphs to each vertex of 
$G$ is unmixed and vertex decomposable, in particular, Cohen--Macaulay. 
This is a generalization of the result 
by Villarreal \cite[Proposition 2.2]{Villarreal90} and Dochtermann and Engstr\"om \cite[Theorem 4.4]{DE}. 

We recall the definition of a 
vertex decomposable simplicial complex. 
For a simplicial complex $\Delta$ and its vertex $x$, 
let $\Delta \setminus x=\{\sigma \in \Delta : x \not\in \sigma\}$ and 
$\link_\Delta(x)=\{\sigma \in \Delta : \sigma \cap \{x\} = \emptyset, \sigma \cup \{x\} \in \Delta\}$. 
A simplicial complex $\Delta$ is called {\em vertex decomposable} if $\Delta$ is a simplex, 
or if there exists a vertex $x$ of $\Delta$ such that 
\begin{itemize}
\item[(i)] $\Delta \setminus x$ and $\link_\Delta(x)$ are vertex decomposable and 
\item[(ii)] no face of $\link_\Delta(x)$ is a facet of $\Delta \setminus x$. 
\end{itemize}

We also recall the definition of a 
shellable simplicial complex. 
A simplicial complex $\Delta$ is called \textit{shellable} if all of its facets can be listed \[F_1,\ldots,F_s\] in such a way that 
\[\left( \bigcup_{j=1}^{i-1} \langle F_j \rangle \right) \cap \langle F_i \rangle =\bigcup_{j=1}^{i-1} \langle F_j \cap F_i \rangle\]
is a pure simplicial complex of dimension $\dim F_i -1$ for every $1 < i \leq s$. 
Here $\langle F_i \rangle=\{ \sigma \in \Delta : \sigma \subset F_i\}$. 
When this is the case, we call the order $F_1, \ldots, F_s$ a 
\textit{shelling} of $\Delta$. 

For simplicial complexes, the following implications are known: 
\begin{itemize}
\item vertex decomposable $\Longrightarrow$ shellable $\Longrightarrow$ sequentially Cohen--Macaulay; 
\item pure and vertex decomposable $\Longrightarrow$ pure and shellable $\Longrightarrow$ Cohen--Macaulay. 
\end{itemize}

\par
Let $G$ be a graph on the vertex set $V$. 
For $W \subset V$, we denote by $G \setminus W$, the induced subgraph of $G$ 
on $V \setminus W$. For a vertex $v \in V$, we use notation $G \setminus v$ instead of $G \setminus \{v\}$. 
For $v \in V$, let $N(v)$ (or $N_G (v)$) denote the neighbourhood of $v$ in $G$ 
and let $N[v]=N(v) \cup \{v\}$.

\par
Let $G$ be a finite simple graph on a vertex set $V$. 
Villarreal \cite[Proposition 2.2]{Villarreal90} proved that 
the graph obtained from $G$ by adding a whisker to each vertex is Cohen--Macaulay. 
Dochtermann and Engstr\"om \cite[Theorem 4.4]{DE} proved that 
such a graph is unmixed and vertex decomposable. 
Adding a whisker to each vertex is the same as saying that 
attaching the complete graph $K_2$ to each vertex. 
We generalize the above results as follows. 
\begin{theorem}
  \label{CM}
  Let $G$ be a finite simple graph on a vertex set 
  $V = \{ x_1, \ldots, x_n \}$. 
  Let $k_1, \ldots, k_n \geq 2$ be integers. 
  Then the graph $G'$ obtained from $G$ 
  by attaching the complete graph $K_{k_i}$ 
  to $x_i$ for $i=1, \ldots, n$ is unmixed and vertex decomposable. 
  In particular, $G'$ is shellable and Cohen--Macaulay. 
\end{theorem}
Before proving Theorem \ref{CM}, we recall the result by Cook and Nagel \cite[Theorem 3.3]{CN}, 
which is another generalization of \cite[Proposition 2.2]{Villarreal90} and \cite[Theorem 4.4]{DE}. 
A {\em clique vertex-partition} of $G$ is a set $\pi=\{W_1,\ldots,W_t\}$ 
of disjoint (possibly empty) cliques of $G$ such that their disjoint union forms $V$. 
For a clique vertex-partition $\pi=\{W_1,\ldots,W_t\}$, let $G^\pi$ denotes the graph 
on the vertex set $V \cup \{w_1,\ldots,w_t\}$ with the edge set 
$E(G) \cup \bigcup_{i=1}^t \{\{v,w_i\} : v \in W_i\}$. 
\begin{lemma}[Cook and Nagel \cite{CN}]\label{vd}
Let $\pi=\{W_1,\ldots,W_t\}$ be a clique vertex-partition of $G$. 
Then $G^\pi$ is unmixed and vertex decomposable. 
\end{lemma}

Now we prove Theorem \ref{CM}. 
\begin{proof}[Proof of Theorem \ref{CM}]
For $i=1,\ldots,n$, 
let $\{w_j^{(i)} : j=1,\ldots,k_i\}$ be the vertex set of the attached complete graph $K_{k_i}$, 
where $w_1^{(i)}=x_i$. 
We consider the graph $H$ obtained from $G$ by attaching the complete graph $K_{k_i-1}$ 
to $x_i$ for $i=1,\ldots,n$ whose vertex set is $\{w_1^{(i)},\ldots,w_{k_i-1}^{(i)}\}$, 
that is, $H=G' \setminus \{w_{k_i}^{(i)} : i=1,\ldots,n\}$. 
Then $H$ has a clique vertex-partition $\pi=\{W_1,\ldots,W_n\}$, 
where $W_i=\{w_1^{(i)},\ldots,w_{k_i-1}^{(i)}\}$, and $G'=H^\pi$. 
By Lemma \ref{vd}, we conclude that $G'=H^\pi$ is unmixed and vertex decomposable, as desired. 
\end{proof}

\par
Now we classify all Cohen--Macaulay Cameron--Walker graphs. 
Actually we have the following theorem for Cameron--Walker graphs. 
\begin{theorem}
\label{Kimura}
  For a Cameron--Walker graph $G$, 
  the following five conditions are equivalent: 
  \begin{enumerate}
  \item $G$ is unmixed. 
  \item $G$ is Cohen--Macaulay. 
  \item $G$ is unmixed and shellable. 
  \item $G$ is unmixed and vertex decomposable. 
  \item $G$ consists of a connected bipartite graph with vertex partition 
    $[n] \sqcup [m]$ such that 
    there is exactly one leaf edge attached to each vertex
    $i \in [n]$ and that there is exactly one pendant triangle
    attached to each vertex $j \in [m]$.
  \end{enumerate}
\end{theorem}
\begin{proof}
  $(5) \Longrightarrow (4)$ follows from Theorem \ref{CM}. 
  As mentioned above, 
  $(4) \Longrightarrow (3)$ and $(3) \Longrightarrow (2)$ follow. 
  $(2) \Longrightarrow (1)$ is well known. 
  We prove $(1) \Longrightarrow (5)$. 

  \par
  Since $G$ is a Cameron--Walker graph, 
  $G$ consists of a connected bipartite graph 
  with vertex partition $[n] \sqcup [m]$ such that 
  there is at least one leaf edge attached to each vertex $i \in [n]$ 
  and that there may be possibly some pendant triangles 
  attached to each vertex $j \in [m]$. 
  Assume that there are $f$ leaves in total and $t$ total pendant triangles. 
  Also assume that out of each vertex $j \in [m']$ ($m' \leq m$), 
  there is at least one 
  triangle and there is no triangle out of each vertex 
  $j \in \{ m'+1, \ldots, m \}$. 

  \par
  The following three subsets of the vertex set of $G$ 
  are minimal vertex covers of $G$: 
  \begin{enumerate}
  \item[(i)] the subset $[n]$ of the vertex set 
    of the connected bipartite graph and 
    two vertices of degree $2$ of each pendant triangle; 
  \item[(ii)] the subset $[m]$ of the vertex set 
    of the connected bipartite graph, 
    all leaves 
    and exactly one vertex of degree $2$ of each pendant triangle; 
  \item[(iii)] the subset $[n] \sqcup [m']$ of the vertex set 
    of the connected bipartite graph 
    and exactly one vertex of degree $2$ of each pendant triangle. 
  \end{enumerate}
  The cardinalities are (i) $n + 2t$; (ii) $m+f+t$; (iii) $n + m' + t$. 
  Since $G$ is unmixed, these are equal. 
  By (i) and (iii), we have $t=m'$; it follows that there is just one 
  triangle out of each vertex $j \in [m']$. 
  By (i) and (ii), we have $n+t = m+f$, i.e., $m-t = n-f$. 
  Since there is at least one leaf out of each vertex $i \in [n]$, we have 
  $n -f \leq 0$. On the other hand $m-t = m-m' \geq 0$. 
  Therefore we have $m-t = n-f =0$. 
  Then we have the desired assertion. 
\end{proof}


\section{Gorenstein Cameron--Walker graphs} 
In this section, we consider the problem which Cohen--Macaulay Cameron--Walker graph is Gorenstein. 
In order to attack the problem, we compute the Cohen--Macaulay type of a Cohen--Macaulay Cameron--Walker graph. 

\par
Let $G$ be a graph on the vertex set $V$. 
We say that a subset $A \subset V$ is an {\em independent set} if no two vertices in $A$ are adjacent in $G$. 

\par
The following theorem is the main result in this section. 
\begin{theorem}\label{GorCWgraph}
Let $G$ be a Cohen--Macaulay Cameron--Walker graph with $m$ pendant triangles. 
Then the Cohen--Macaulay type of $G$ is equal to $2^m$. In particular, there is no Gorenstein Cameron--Walker graph. 
\end{theorem}
\begin{proof}
Let $G$ be a Cohen--Macaulay Cameron--Walker graph. 
Then $G$ is a graph described in Theorem 1.3 (5); 
let $[x_1,\ldots,x_n] \sqcup [y_1,\ldots,y_m]$ be the vertex partition 
of the supporting connected bipartite graph of $G$, 
$v_i$ the leaf of $G$ attached to $x_i$ 
and $z_j, w_j$ the two vertices of degree 2 
which form the pendant triangle with $y_j$. 
Then $I(G)$ is an ideal of 
$S=K[\{x_i,v_i: 1 \leq i \leq n\} \cup \{y_j,z_j,w_j : 1 \leq j \leq m\}]$. 
(We identify each vertex of $G$ with the variable of $S$.) 
Since the cardinality of the minimal vertex cover is equal to 
$\height I(G)$, 
it follows from the proof of Theorem 1.3 that $\height I(G)=n+2m$. 
Thus, we obtain that 
$\dim(S/I(G))=\dim S - \height I(G)= n+m$. 


\par
Consider the sequence 
\begin{displaymath}
  {\bf x}=x_1-v_1,\ldots,x_n-v_n,w_1-(y_1+z_1),\ldots,w_m-(y_m+z_m) 
\end{displaymath}
of $n+m$ elements of $S$. Let us consider the polynomial ring 
$S'=K[\{x_i : 1 \leq i \leq n\} \cup \{y_j,z_j : 1 \leq j \leq m \}]$. 
Set 
$$J' :=(x_i^2 : 1 \leq i \leq n) + (y_j^2,z_j^2 : 1 \leq j \leq m) 
\subset S'.$$ 
Also let $G'$ be the induced subgraph of $G$ on 
$\{x_i : 1 \leq i \leq n\} \cup \{y_j,z_j : 1 \leq j \leq m\}$. 
Then $I(G')$ is an ideal of $S'$. 
Modulo the sequence {\bf x}, one has $S/(I(G)+({\bf x})) \cong S'/(J'+I(G'))$, 
and thus $\dim (S/(I(G)+({\bf x})))= 0$. 
Since $\dim (S/I(G))=n+m$ and the sequence {\bf x} consists of $n+m$ elements, 
{\bf x} is a linear system of parameter. 
Hence {\bf x} is a regular sequence of $S/I(G)$ 
because $S/I(G)$ is Cohen--Macaulay.


\par
The Cohen--Macaulay type of $S/I(G)$ coincides with 
$\dim_K \Soc (S/(I(G)+({\bf x})))$ 
(\cite[Proposition A.6.1]{HHgtm260}). 
Therefore, we compute 
$\dim_K \Soc (S/(I(G)+({\bf x}))) = \dim_K \Soc (S'/(J'+I(G')))$. 
Set $T=S'/(J'+I(G'))$. 
Since $\Soc (T)=\{x \in T : {\bf m}x=0\}$, 
where ${\bf m}=(x_i : 1 \leq i \leq n)+(y_j,z_j : 1 \leq j \leq m)$, 
a set of elements $v+J'+I(G')$, 
where $v$ is a monomial in $S'$, 
such that $v \not\in J'+I(G')$ and ${\bf m} v \subset J'+I(G')$ 
forms a basis for the $K$-vector space $\Soc (T)$. 
By Lemma \ref{hodai} below, one can compute $\dim_K \Soc (T)$ by counting the maximal independent sets of $G'$. 
%
%

\par
A maximal independent set $A$ of $G'$ is uniquely determined 
by the intersection $A \cap \{ y_1, \ldots, y_m \}$. 
In fact, for a subset $Y$ of $\{ y_1, \ldots, y_m \}$, 
there exists a unique maximal independent set of $G'$: 
$Y \cup (\{x_1,\ldots,x_n\} \cup \{z_1,\ldots,z_m\} 
\setminus \bigcup_{y \in Y}N(y))$. 
Hence it follows that there are exactly $2^m$ maximal independent sets of $G'$, 
as desired. 
%
%
%
%
%
\end{proof}

\begin{lemma}\label{hodai}
With the same notation as in the proof of Theorem \ref{GorCWgraph}, 
there is a one-to-one correspondence between each monomial $v$ in $S'$ 
such that $v \not\in J'+I(G')$ and ${\bf m} v \subset J'+I(G')$ 
and each maximal independent set of $G'$. 
\end{lemma}
\begin{proof}
Take a monomial $v=s_1\cdots s_\ell$ in $S'$ 
such that $v \not\in J'+I(G')$ and ${\bf m} v \subset J'+I(G')$. 
Then $s_1,\ldots,s_\ell$ are distinct elements 
of $\{x_i : 1 \leq i \leq n\} \cup \{y_j,z_j : 1 \leq j \leq m\}$ 
because $v \not\in J'$. 
Moreover, since $v \not\in I(G')$, 
it follows that $\{s_1,\ldots,s_\ell\}$ is an independent set of $G'$. 
Now we prove that $\{ s_1, \ldots, s_\ell \}$ is maximal. 
Take 
$x \in \{x_i : 1 \leq i \leq n\} \cup \{y_j,z_j : 1 \leq j \leq m\}$ 
with $x \not\in \{s_1,\ldots,s_\ell\}$. 
Since ${\bf m} v \subset J'+I(G')$, we have $xv \in I(G')$. 
This means that there is $1 \leq k \leq \ell$ such that 
$\{x,s_k\}$ is an edge of $G'$. 
Hence, $\{s_1,\ldots,s_\ell\}$ is a maximal independent set. 

\par
On the other hand, take a maximal independent set 
$\{s_1,\ldots,s_\ell\}$ of $G'$. 
Then the corresponding squarefree monomial $v := s_1\cdots s_\ell$ 
does not belong to $J'+I(G')$. 
Also take $x \in \{x_i : 1 \leq i \leq n\} \cup \{y_j,z_j : 1 \leq j \leq m\}$. 
If $x \in \{ s_1, \ldots, s_\ell \}$, then $xv \in J'$. 
Otherwise, $xv \in I(G')$ because 
$\{s_1,\ldots,s_\ell\}$ is a maximal independent set of $G'$. 
Therefore $xv \in J' + I(G')$. 
\end{proof}


\section{Sequentially Cohen--Macaulayness of Cameron--Walker graphs}
In this section, we prove that every Cameron--Walker graph is 
sequentially Cohen--Macaulay. 
Actually, we prove that it is vertex decomposable and thus, shellable. 
We also provide a shelling for a Cameron--Walker graph 
whose supporting connected bipartite graph is a complete bipartite graph. 

The following theorem is the main result in this section. 
\begin{theorem}
  \label{seqCM}
  Every Cameron--Walker graph is vertex decomposable, 
  in particular, shellable and sequentially Cohen--Macaulay. 
\end{theorem}

To prove Theorem \ref{seqCM}, we use the following. 
\begin{lemma}[{Woodroofe \cite[Theorem 1]{Woodroofe09}}]
  \label{Woodroofe}
  Let $G$ be a graph with no chordless cycles of length other than $3$ or $5$. 
  Then $G$ is vertex decomposable. 
\end{lemma}

The following is a rephrase of the definition of the vertex decomposability of graphs. 

\begin{lemma}[{cf. \cite[Lemma 4]{Woodroofe09}}]\label{vd2}
A finite simple graph $G$ is vertex decomposable if and only if 
$G$ is totally disconnected (i.e., $G$ has no edge), or if 
there is a vertex $v$ of $G$ such that 
\begin{itemize}
\item[(i)'] $G \setminus v$ and $G \setminus N[v]$ are vertex decomposable and 
\item[(ii)'] no independent set in $G \setminus N[v]$ is a maximal independent set in $G \setminus v$. 
\end{itemize}
\end{lemma}

Now we prove Theorem \ref{seqCM}. 
\begin{proof}[Proof of Theorem \ref{seqCM}]
Let $G$ be a Cameron--Walker graph on the vertex set $V$ whose supporting connected bipartite graph 
has a vertex partition $[n] \sqcup [m]$ such that there is at least one leaf edge 
attached to each vertex $i \in [n]$ and that there may be possibly some pendant triangles 
attached to each vertex $j \in [m]$. 
We prove the assertion by induction on $n$. 

When $n=1$, the supporting bipartite graph is a star. 
Thus $G$ contains no cycle except for pendant triangles. 
Hence, by Lemma \ref{Woodroofe}, $G$ is vertex decomposable. 

Assume that $n>1$. 
We take the vertex $n \in [n]$ 
and consider the graphs $G_1=G \setminus n$ and $G_2 = G \setminus N[n]$. 
Since a union of an independent set of $G \setminus N[n]$ and a leaf adjacent to $n$ 
is an independent set of $G \setminus n$, the condition (ii)' in Lemma \ref{vd2} is satisfied with $v=n$. 
Therefore, in order to prove that $G$ is vertex decomposable, it is sufficient to prove that 
$G_1$ and $G_2$ are vertex decomposable by Lemma \ref{vd2}. 
Also, in order to prove that $G_1$ and $G_2$ are vertex decomposable, it is sufficient 
to prove that all of their connected components are vertex decomposable (see \cite[Lemma 20]{Woodroofe09}). 
Each of connected components of $G_1$ and $G_2$ is one of 
the following four graphs: 
\begin{itemize}
\item[(i)] an isolated vertex; 
\item[(ii)] an edge; 
\item[(iii)] a triangle; 
\item[(iv)] a Cameron--Walker graph whose supporting connected 
bipartite graph has a vertex partition 
$[n'] \sqcup [m']$ with $n'<n$ and $m' \leq m$. 
\end{itemize}
Clearly, the first three graphs (i), (ii) and (iii) are vertex decomposable. 
Moreover, by the inductive hypothesis, the graph (iv) is also vertex decomposable, as desired. 
\end{proof}

\begin{remark}
In \cite{BC}, Biyiko\u{g}lu and Civan also treat the graph $G$ satisfying $im(G)=m(G)$. 
They prove that $G$ is codismantlable if $im(G)=m(G)$ (\cite[Theorem 3.7]{BC}). 
They study the relation between the codismantlability and the vertex decomposability of graphs. 
For more precise information, see \cite[Section 2]{BC}. 
\end{remark}

\bigskip

Next we consider the Cameron--Walker graphs whose supporting connected 
bipartite graphs are complete bipartite graphs. 
We provide a shelling for these graphs though we have already known that 
these are shellable by Theorem \ref{seqCM}.


\par
Let $G$ be a Cameron--Walker graph on the vertex set $V$ whose supporting 
connected bipartite graph is $K_{n,m}$ with vertex partition $[n]\sqcup[m]$ 
such that out of each vertex $i\in[n]$ there is at least one leaf and out of each vertex $j\in[m]$ there may be a pendant triangle. Assume that there are $f$ leaves in total and out of each vertex $i\in[n]$ there are $f_i$ leaves so that $\sum_{i=1}^nf_i=f$. Furthermore assume that there are $t$ total pendant triangles and that out of each vertex $j\in[m]$ there are $t_j\geq 0$ triangles so that $t=\sum_{j=1}^mt_j$. Let $m'\leq m$ be the number of vertices in $[m]$ with at least one adjacent pendant triangle. Then by relabelling we may assume $t_j=0$ for $j>m'$.
We will use notation by referring to vertices of $G$ as in Figure 2. 

\bigskip

\begin{center}
\includegraphics[scale=1.2]{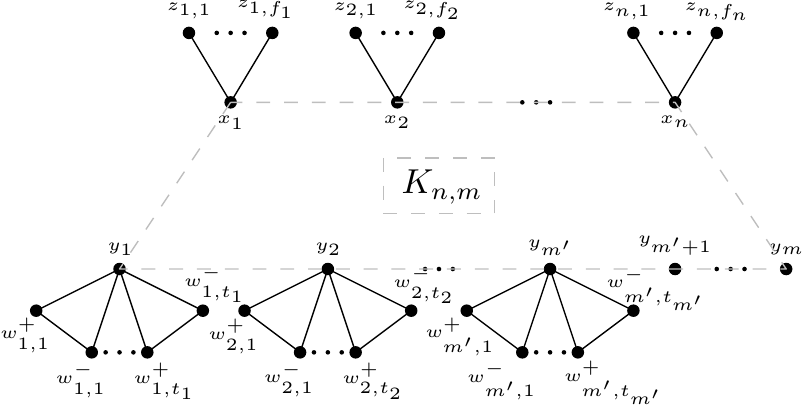}

\bigskip
{\bf Figure 2} (the vertices of $G$)

\end{center}

\bigskip


Since the vertices $x_1,\ldots,x_n,y_1,\ldots,y_m$ form a complete bipartite graph $K_{n,m}$, 
it is easy to check that the facets of the clique complex 
$\Delta(\overline{G})$ are of one of the following two forms; 
note that a clique of $\overline{G}$ 
is equivalent to an independent set in $G$: 
\begin{enumerate}[(I)]
\item $\displaystyle \{y_{m'+1},\dots,y_m\}\cup\{z_{k,l} : 1\leq k \leq n, \, 1\leq l \leq f_k\}\cup\{y_i : i \in I \} \cup \{w_{i',j}^{\epsilon_{i',j}} : i'\in[m'] \setminus I, \, 1\leq j \leq t_{i'}\}$ 
where $I \subset [m']$ and $\epsilon_{i',j}\in\{+,-\}$ for each $i', j$; 
\item $\displaystyle \{x_j : j \in J \} \cup \{ z_{j',i} : j'\in[n]\setminus J, \, 1\leq i \leq f_{j'}\} \cup \{w_{k,l}^{\epsilon_{k,l}} : 1\leq k\leq m', \, 1\leq l \leq t_k\}$ 
where $\emptyset \neq J \subset [n]$ 
and $\epsilon_{k,l}\in\{+,-\}$ for each $k, l$. 
\end{enumerate}
Note that each $I\subset[m']$ defines a family of facets of the first form, which we will denote $\mathcal{F}_I$. Similarly, each $\emptyset \neq J\subset[n]$ defines a family of the second form, $\mathcal{G}_J$. For a fixed $I\subset[m']$ note that all facets $F\in\mathcal{F}_I$ contain 
$\{y_{m'+1},\dots,y_m \} \cup \{ z_{k,l} : 1\leq k\leq n, \, 1\leq l\leq f_k\}$ and so $F$ is determined completely by the  $w_{i',j}^{\epsilon_{i',j}}$ in $F$. Thus we have
\[
\mathcal{F}_I=\{F_\nu ~:~ \nu\in\{+,-\}^{s_I}\}
\] 
where $s_I=\sum_{i\in[m']\setminus I}t_i.$
Similarly, we have for $\emptyset \neq J \subset [n]$
\[
\mathcal{G}_J=\{G_\nu ~:~ \nu\in\{+,-\}^t\}.
\]

We first order facets within $\mathcal{F}_I$ (resp.\  $\mathcal{G}_J$). 
Consequently, we have that the simplicial complex with facets $\mathcal{F}_I$ 
(resp.\  $\mathcal{G}_J$) is shellable. 
To do this, we determine a total order on $\{\nu : \nu \in \{+,-\}^{s_I}\}$ 
which will induce the order for our shelling. 
Let $\nu_+$ be the number of $+$'s in $\nu$ and $\nu^{(k)}$ denote the $k^\text{th}$ entry of $\nu$. We say that $\nu<\omega$ if and only if either of the following two conditions hold
\begin{enumerate}[(i)]
\item $\nu_+<\omega_+$, or 
\item $\nu_+=\omega_+$, $\nu^{(d)}=+$, and $\omega^{(d)}=-$ where $d$ is the first entry $\nu$ and $\omega$ differ when reading from the left.
\end{enumerate} 
For example
\[
(-,-,-)<(+,-,-)<(-,+,-)<(-,-,+)<(+,+,-)<\cdots<(+,+,+).
\]
\begin{lemma}\label{FamilyShell}
  For each fixed $I\subset[m']$, 
  \[
  F_{\nu_1}, F_{\nu_2}, \dots, F_{\nu_{2^{s_I}}}
  \]
  is a shelling for the simplicial complex 
  with facets $\mathcal{F}_I$, 
  where $\nu_{k}>\nu_{k+1}$ for all $1\leq k\leq 2^{s_I}-1$. 

  A similar order on $\{ \nu : \nu \in \{ +, - \}^t \}$ gives a shelling 
  for the simplicial complex with facets $\mathcal{G}_J$ 
  for each fixed $\emptyset \neq J \subset [n]$. 
\end{lemma}
\begin{proof}
We first note that all $F_{\nu_i} \in \mathcal{F}_I$ have the same dimension, 
denoted by $d_I$. 
To show the assertion, we need merely verify that 
\[
\left( \bigcup_{k=1}^{l-1} \langle F_{\nu_k} \rangle \right) \cap \langle F_{\nu_l} \rangle = 
\bigcup_{k=1}^{l-1}\langle F_{\nu_k} \cap F_{\nu_l} \rangle
\]
is a pure simplicial complex of dimension $d_I-1$. 

Suppose that $1\leq k<l\leq 2^{s_I}$ such that $\dim (F_{\nu_k}\cap F_{\nu_l}) < d_I-1$. Let $R$ be a subset of $[s_I]$ such that $\nu_k^{(r)}=\nu_l^{(r)}$ if and only if $r\in R$. Note that $0\leq |R|\leq s_I-2$. Clearly $\nu_l^{(u)}=-$ for some $u\in[s_I]\setminus R$ otherwise we would not have $\nu_k>\nu_l$. 
Let $k'<k$ be the index such that $\nu_{k'}^{(u)}=+$ and agrees with $\nu_l$ in every other entry. Then $F_{\nu_l}\cap F_{\nu_k}\subset F_{\nu_l}\cap F_{\nu_{k'}}$ and $\dim(F_{\nu_l}\cap F_{\nu_{k'}})=d_I-1.$

A similar argument shows the assertion for $\mathcal{G}_J$. 
\end{proof}

Now we define an order on the all facets of $\Delta (\overline{G})$. 

\par
We define an order on the families of facets, $\mathcal{F}_I$ with $I\subset[m']$ (respectively $\mathcal{G}_J$ with $\emptyset \neq J\subset[n]$) by imposing a total order on the subsets $I\subset[m']$ (respectively $\emptyset \neq J\subset[n]$). For any subsets $I,L\subset[m']$, we say that $L<I$ if and only if either $|I|<|L|$ or $|I|=|L|$ and the first non-zero entry of the vector
\[
\sum_{i\in I}\eb_i-\sum_{l\in L}\eb_l
\]
reading from left to right is a $1$. Here we take $\eb_i \in \nR^{m'}$ to be the standard basis vectors.
For example
\[
\{1,3,4\}<\{1,2,3\}<\{2,4\}<\{5\}<\emptyset.
\]
We define a similar order for all subsets $\emptyset \neq J\subset[n]$. 


\begin{proposition}\label{shellable}
Let $G$ be a Cameron--Walker graph whose supporting connected bipartite graph 
is a complete bipartite graph. 
Then 
\[
\mathcal{F}_\emptyset,\mathcal{F}_{\{1\}},\mathcal{F}_{\{2\}},\dots,\mathcal{F}_{[m']},\mathcal{G}_{\{1\}},\mathcal{G}_{\{2\}},\dots,\mathcal{G}_{[n]}
\]
is a shelling of $\Delta (\overline{G})$, 
where the order of the indexing sets is given by the order defined above and the order within each family $\mathcal{F}_I$ or $\mathcal{G}_J$ is given by the shelling order in Lemma \ref{FamilyShell}.
\end{proposition}
\begin{proof}
Let $I,L\subset[m']$ with $I>L$ and let $F_\nu\in\mathcal{F}_I$ and $F_\omega\in\mathcal{F}_L$ such that $\dim(F_\nu\cap F_\omega)<\dim F_\omega-1=d_L-1$. We then proceed as in the proof of Lemma \ref{FamilyShell} to show that there exists $F_{\nu'}\in\mathcal{F}_{I'}$ with $I' > L$ such that $F_\nu\cap F_\omega\subset F_{\nu'}\cap F_\omega$ and $\dim (F_{\nu'}\cap F_\omega) = d_L-1$. 
Since $I > L$, one has $L \not\subset I$, i.e., $L \setminus I \not= \emptyset$. Choose $k \in L \setminus I$ and set $I'=L \setminus \{k\}$. 
Thus we have $|I'|=|L|-1$. Take a vector $\bar{\omega}\in\{+,-\}^{s_{I'}}$ as an extension of $\omega\in\{+,-\}^{s_L}$. 
Then $F_{\bar{\omega}}\cap F_{\omega}=F_{\omega}\setminus\{y_k\}.$ Furthermore, our choice of $k$ guarantees that $F_\nu\cap F_\omega\subset F_{\bar{\omega}}\cap F_\omega.$ A similar argument holds for $\emptyset \neq J,L\subset[n]$ with $J>L$, $G_\nu\in\mathcal{G}_J$, and $G_\omega\in\mathcal{G}_L$. In this case we consider $G_\omega\in\mathcal{G}_{J'}$ where $J'=L\setminus\{k\}$ for some $k \in L \setminus J$ and proceed as above.

Finally, let $I\subset[m']$, $\emptyset \neq J\subset[n]$, $F_\nu\in\mathcal{F}_I$, and $G_\omega\in\mathcal{G}_J$ such that $\dim(F_\nu\cap G_\omega)<\dim G_\omega-1$. First we observe that $x_k\not\in F_\nu\cap G_\omega$ for all $k\in[n]$. Similarly $y_l\not\in F_\nu\cap G_\omega$ for all $l\in[m]$. 
If $|J|=1$, i.e., $J=\{j\}$ for some $j\in [n]$, then we consider $F_{\omega}\in\mathcal{F}_\emptyset$ where the subscript $\omega$ for $F_\omega$ denotes the same sign vector as that of $G_\omega$. (Here we note $s_\emptyset = t$.) Then $F_\nu\cap G_\omega\subset F_\omega\cap G_\omega = G_\omega \setminus\{x_j\}$. 
If $|J|>1$, then let $J'\subset J$ such that $|J'|=|J|-1$. It is then easy to check that $F_\nu\cap G_\omega\subset G'_{\omega}\cap G_\omega$ where $G'_{\omega}\in\mathcal{G}_{J'}$ and furthermore $\dim(G'_{\omega}\cap G_\omega)=\dim(G_\omega)-1$.
\end{proof}

We finish the paper by noting that we can describe the projective dimension $\pd (K[V]/I(G))$ 
for a Cameron--Walker graph $G$ on $V$ in terms of the graph. 
Let $i(G)$ be a minimum cardinality of independent sets $A$ with $A \cup N_G(A)=V$. 
Since a Cameron--Walker graph is sequentially Cohen--Macaulay by Theorem \ref{seqCM}, 
we have the following corollary by the results of Dao and Schweig \cite{DS}. 
\begin{corollary}[{\cite[Corollary 5.6 and Remark 5.7]{DS}}]
Let $G$ be a Cameron--Walker graph on the vertex set $V$. Then $\pd(K[V] / I(G))= |V|-i(G)$. 
\end{corollary}


\end{document}